\documentclass[11pt]{amsart}
\usepackage{mathtools,amsthm,amssymb,amscd,amsrefs}
\usepackage[usenames]{color}
\usepackage{tikz}
\usetikzlibrary{shapes,decorations.pathreplacing,arrows}
\usepackage[enableskew]{youngtab}
\usepackage[colorlinks=true, linkcolor=blue, citecolor=blue, urlcolor=blue]{hyperref}

\DeclareMathOperator{\Gr}{Gr}
\newcommand{\IG}{\mathrm{IG}}
\newcommand{\OG}{\mathrm{OG}}
\newcommand{\QH}{\mathrm{QH}}
\DeclareMathOperator{\ev}{ev}

\newtheorem{thm}{Theorem}[section]
\newtheorem{lemma}[thm]{Lemma}

\newtheorem{prop}[thm]{Proposition}

\theoremstyle{definition} 
\newtheorem{defn}[thm]{Definition}

\theoremstyle{remark}
\newtheorem{remark}[thm]{Remark}
\newtheorem{example}[thm]{Example}

\begin{document}
\title{Minimum quantum degrees for Isotropic Grassmannians in types B and C}

\author{Ryan M. Shifler}
\address{
Department of Mathematical Sciences,
Henson Science Hall, 
Salisbury University,
Salisbury, MD 21801
}
\email{rmshifler@salisbury.edu}

\author{Camron Withrow}
\address{
Department of Mathematics,
McBryde Hall,
Virginia Tech,
Blacksburg, VA 24061
}
\email{cwithrow@vt.edu}

\begin{abstract}
We give a formula in terms of Young diagrams to calculate the minimum positive integer $d$ such that $q^d$ appears in the quantum product of two Schubert classes for the submaximal isotropic Grassmannians in types B and C. We do this by studying curve neighborhoods. We compute curve neighborhoods in several combinatorial models including $k$-strict partitions and a set of partitions where their inclusion is compatible with the Bruhat order.
\end{abstract}

\maketitle

\section{Introduction}\label{sect:intro}
Let $\Gr:=\Gr(k,n)$ be the Grassmannian, $\IG:=\IG(k,2n)$ be the symplectic Grassmannian, and $\OG:=\OG(k,2n+1)$ be the odd orthogonal Grassmannian. Let $\QH^*(X)$ be the quantum cohomology ring of $X \in \{ \Gr, \IG, \OG \}$. The purpose of this paper is to give a formula in terms of inclusions of Young diagrams to calculate the minimum positive integer $d$ such that $q^d$ appears in the quantum product of two Schubert classes in $\QH^*(X)$. In particular, we will heavily focus on the submaximal isotropic Grassmannians in types B and C (i.e. $\OG(k,2n+1)$ and $\IG(k,2n)$ when $k<n$).

The minimum quantum degree is calculated using Young diagrams for $\Gr$ by Postnikov \cites{postnikov:qbruhat, postnikov:affine} and Fulton and Woodward \cite{FW}. See also \cites{Buch,yong, Belkale}. In a recent paper B\"arligea \cite{Bar} gives an explicit formula to compute the minimum degree of the point class times the point class in terms of the cascade of orthogonal roots in $G/P$. We approach this problem by considering curve neighborhoods and various combinatorial models.

We study two sets of partitions that index Schubert varieties of $\OG$ and $\IG$. The first is the set of $(n-k)$-strict partitions given by
\[ 
\Lambda:=\{(\lambda_1 \geq \ldots \geq \lambda_k): 2n-k \geq \lambda_1, \lambda_k \geq 0, \mbox{ and if } \lambda_j>n-k \mbox{ then } \lambda_{j+1}<\lambda_j   \}.
\] 
The $(n-k)$-strict partitions are defined and used to calculate a Pieri rule for isotropic Grassmannians in \cites{BKT,BKT2}. For the next set of partitions first recall the set of partitions that index Schubert varieties (by codimension) in the Grassmannian $\Gr$. It is the set $\mathcal{P}(k,n)$ given by \[\{ (\lambda_1 \geq \lambda_2 \geq \cdots \geq \lambda_k): n-k \geq \lambda_1,\lambda_k \geq 0  \}.\]

The Schubert varieties of $\OG$ and $\IG$ are indexed by a subset of $\mathcal{P}(k,2n)$. We give the relevant definitions next.

\begin{defn}
Let $\lambda \in \mathcal{P}(k,n)$ be a partition. This partition's boundary consists of $n$ steps moving either left or down in the south-west direction. Let $D(\lambda)(i)=0$ if the $i$th step is left and $D(\lambda)(i)=1$ if the $i$th step is down.
\end{defn}

\begin{defn} \label{def:partbruh} 
Define \[ \mathcal{P}'(k,2n)=\{\lambda \in  \mathcal{P}(k,2n): \mbox{ if } D(\lambda)(i)=D(\lambda)(2n+1-i) \mbox{ then } D(\lambda)(i)=0 \}.\]
\end{defn}

\begin{example}
The partition $(5,2,1) \in \mathcal{P}'(3,8)$.  Pictorially, \[\yng(5,2,1) \in \mathcal{P}'(3,8).  \] On the contrary, the partition $(5,5,1) \notin \mathcal{P}'(3,8)$. Pictorially, \[\yng(5,5,1) \notin \mathcal{P}'(3,8).  \] 

In other words, $D((5,2,1))=10001010$ and $D((5,5,1))=1\underline{1}0000\underline{1}0$. The underlined entries of $D((5,5,1))$ do not satisfy the conditions of $\mathcal{P}’(3,8)$.
\end{example}

The partitions $\Lambda$ and $\mathcal{P}'(k,2n)$ each have a desired property. When using $(n-k)$-strict partitions $\Lambda$, the codimension of the Schubert variety $X^\lambda$ in $\IG$ and $\OG$ is $|\lambda|=\lambda_1+\lambda_2+\cdots+\lambda_k$. While partition inclusion and the Bruhat order are compatible for partitions in $\Lambda$ in the cominuscule cases (i.e. $k=n$), the compatibility fails to hold in the submaximal cases. On the other hand, partition inclusion on the set $\mathcal{P}'(k,2n)$ {\it does} respect the Bruhat order (this is  Proposition \ref{prop:partin}) but the partitions do not index the (co)dimension of Schubert varieties. 

Next we discuss curve neighborhoods. Let $X \in \{\Gr,\IG, \OG \}$. Let $X^\lambda \subset X$ be a Schubert variety where $\lambda$ is chosen from the appropriate indexing set. The curve neighborhood of $X^\lambda$, denoted by $\Gamma_d^X(X^\lambda)$, is the closure of rational curves of degree $d$ that intersect $X^\lambda$. It is known from \cite{BCMP:qkfin} that $\Gamma_d^X(X^\lambda)$ is a Schubert variety. So, there exists a partition $\lambda^d$ such that $\Gamma_d^X(X^\lambda)=X^{\lambda^d}$. The notion of the operator $\cdot^d$ is discussed further in Subsection \ref{subsection:mindegpointofview}. In particular, we define the operator $\cdot^d$ on tableaux and other indexing sets such that $\lambda^d$ is the correct index that realizes the aforementioned existence. Curve neighborhoods are calculated in terms of Young diagrams in \cite[Subsection 3.2, Table 1]{BCMP:eqkt} for cominuscule Grassmannians. Curve neighborhoods for the submaximal cases are calculated in terms of the Young diagrams in Theorem \ref{Thm:Ccrv} for $\IG$ and Theorem \ref{Thm:Bcrv} for $\OG$. These calculations are completed by considering a recursively defined formula for curve neighborhoods that is in terms of roots, coroots, and Weyl group elements. This recursion is given by Buch and Mihalcea in \cite{buch.m:nbhds} and it is stated in Proposition \ref{recursion}.

\begin{example}
For an example of Theorem \ref{Thm:Ccrv} let $\lambda \in P'(k,2n)$ index the Schubert variety $X^\lambda \subset \IG$ and let $d$ be an effective degree. Then $\Gamma_d^\IG(X^\lambda)=X^{\lambda^d}$ where
\begin{enumerate}
\item the curve neighborhood is indexed by $\lambda^1=(\lambda_2-1 \geq \lambda_3-1 \geq \cdots \geq \lambda_k-1 \geq 0)$ for $d=1$;
\item and the curve neighborhood is given by the (recursively defined) index $\lambda^d=(\lambda^{d-1})^1$ for $d>1$.
\end{enumerate}
\end{example}

We are using the indexing sets $\Lambda$ and $\mathcal{P}'(k,2n)$ for the first time to calculate curve neighborhoods and minimum degrees in $\IG(k,2n)$ and $\OG(k,2n+1)$. Thus our main results are a specialization of the results from Fulton and Woodward in \cite{FW}. We are ready to state our main results.


\begin{thm} \label{thm:visthmmaxbc}
Let $X \in \{ \IG(n,2n), \OG(n,2n+1)\}$. For any pair of Schubert classes $[X^\lambda], [X_\mu] \in \QH^*(X)$ where $\lambda, \mu \in \Lambda$, the smallest degree $d$ such that $q^d$ appears in  $[X^\lambda] \star [X_\mu]$ with nonzero coefficient is the smallest integer $d$ such that $\lambda^d \subset \mu$.
\end{thm}

\begin{thm} \label{thm:visthm}
Let $X \in \{\OG(k,2n+1), \IG(k,2n)\}$. Here $k<n$ for $\OG$. For any pair of Schubert classes  $[X^\lambda], [X_\mu] \in \QH^*(X)$ where $\lambda, \mu \in \mathcal{P}'(k,2n)$, the smallest degree $d$ such that $q^d$ appears in  $[X^\lambda] \star [X_\mu]$ with nonzero coefficient is the smallest integer $d$ such that $\lambda^d \subset \mu$.
\end{thm}

\begin{example} \label{ex:C}
For $\IG(5,2 \cdot 8)$ consider $\lambda=(11 \geq 11 \geq 11 \geq 4 \geq 4) \in \mathcal{P}'(5,2 \cdot 8)$ and $\mu=(7,7) \in  \mathcal{P}'(5,2 \cdot 8).$ Then we have that 

\[ \lambda=\yng(11,11,11,4,4) \mbox{ and } \mu=\yng(7,7)\]
\begin{align*}
\lambda^1 &= \yng(10,10,3,3)\\
\lambda^2 &= \yng(9,2,2)\\
\lambda^3 &= \yng(1,1).
\end{align*}
Thus, 3 is the smallest degree $d$ such that $q^d$ appears in  $[X^\lambda] \star [X_\mu]$ with nonzero coefficient.
\end{example}

We will discuss the type A Grassmannian case in its entirety since the proofs of the submaximal cases in types B and C will follow a similar structure. The $\Gr(k,n), \IG(n,2n),$ and $\OG(n,2n+1)$ cases are given in Section \ref{sect:typeA}. In Section \ref{sect:typeCpostnikov} we prove Theorem \ref{thm:visthm} for the submaximal isotropic Grassmannian cases in types B and C.

{\em Acknowledgements.} The authors would like to thank Leonardo Mihalcea, Mark Shimozono, and Anders Buch for useful discussions. The first author is partially supported by the Building Research Excellence (BRE) Program at Salisbury University. We would like to thank the anonymous referees for their useful comments and suggestions.

\section{Background}\label{sect:background}

\subsection{Flag varieties}
We will follow the exposition of \cite{BCLM:EulerChar}. Let $X=G/P$ be the flag variety defined by a connected semisimple complex group $G$ and a parabolic subgroup $P$. Fix a maximal torus $T$ and a Borel subgroup $B$ such that $T \subset B \subset P \subset G$. The opposite Borel subgroup $B^- \subset G$ is defined by $B \cap B^-=T$. Let $W=N_G(T)/T$ be the Weyl group of $G$, $W_P=N_P(T)/T$ the Weyl group of $P$, and let $W^P \subset W$ be the subset of minimal-length representatives of the cosets in $W/W^P$. Let $R$ denote the root system of $G$, with positive roots $R^+$ and simple roots $\Delta \subset R^+$. Similarly, let $R^\vee$ denote the coroot system of $G$, with positive coroots $(R^\vee)^+$ and simple coroots $\Delta^\vee \subset (R^\vee)^+$. The parabolic subgroup $P$ is determined by the subset $\Delta_P=\{ \beta \in \Delta: s_{\beta} \in W_P\}.$ Each element $w \in W$ defines a $B$-stable Schubert variety $X_w=\overline{Bw.P}$ and a $B^-$-stable (opposite) Schubert variety $X^w=\overline{B^-w.P}$. If $w \in W^P$ is a minimal-length representative, then $\dim(X_w)=\text{codim}(X^w,X)=\ell(w)$.

The group $H_2(X,\mathbb{Z})=\mathbb{Z}\Delta^\vee/\mathbb{Z}\Delta_{P}^\vee$ is a free $\mathbb{Z}$-module, with a basis consisting of the Schubert classes $[X_{s_\beta}]$ for $\beta \in \Delta \backslash \Delta_P$. Given two elements $d=\sum_\beta d_\beta [X_{s_\beta}]$ and $d'=\sum_\beta d'_\beta [X_{s_\beta}]$ expressed in this basis, we write $d \leq d'$ if and only if $d_\beta \leq d'_\beta$ for each $\beta \in \Delta \backslash \Delta_P$. This defines a partial order on $H_2(X,\mathbb{Z})$.

For any root $\alpha \in R$ that is not in the span $R_P$ of $\Delta_P$, there exists a unique irreducible $T$-invariant curve $X(\alpha) \subset X$ that connects the points $1.P$ and $s_\alpha.P$. An arbitrary irreducible $T$-invariant curve $C \subset X$ has the form $C=w.X(\alpha)$ for some $w \in W$ and $\alpha \in R^+ \backslash R_P$.

\subsection{Quantum cohomology}
Given an effective degree $d \geq 0$ in $H_2(X,\mathbb{Z})$, we let $\overline{\mathcal{M}}_{0,n}(X,d)$ denote the Kontsevich moduli space of $n$-pointed stable maps $f:C \rightarrow X$ of arithmetic genus zero and degree $f_*[C]=d$. This space is equipped with evaluation maps $\text{ev}_i: \overline{\mathcal{M}}_{0,n}(X,d) \rightarrow X$ for $1 \leq i \leq n$, where $\text{ev}_i$ sends a stable map to the image of the $i$-th marked point in its domain. Given cohomology classes $\gamma_1,\ldots,\gamma_n \in H^*(X,\mathbb{Z})$, the corresponding (cohomological) Gromov-Witten invariants of degree $d$ are defined by 
\[ 
\left< \gamma_1, \ldots, \gamma_n\right>_d = \int_{\overline{\mathcal{M}}_{0,n}(X,d)} \text{ev}^*_1(\gamma_1) \wedge \cdots \wedge \text{ev}^*_n(\gamma_n). 
\] 
Let $\mathbb{Z}[q]= \mathbb{Z}[q_\beta: \beta \in \Delta \backslash \Delta_P]$ denote the polynomial ring in variables $q_\beta$ corresponding to the basis elements of $H_2(X,\mathbb{Z})$. Given any degree $d=\sum_\beta d_\beta [X_{s_\beta}] \in H_2(X,\mathbb{Z})$, we will write $q^d=\Pi_\beta q^{d_\beta}_\beta$. The (small) quantum cohomology ring $\text{QH}(X)$ is a $\mathbb{Z}[q]$-algebra which, as a $\mathbb{Z}[q]$-module, is defined by $\text{QH}^*(X)=H^*(X,\mathbb{Z}) \otimes_{\mathbb{Z}} \mathbb{Z}[q]$. The product is defined by 
\[ 
\gamma_1 \star \gamma_2 = \sum_{w,d \geq 0} \left< \gamma_1,\gamma_2, [X_w] \right>_d q^d [X^w] 
\] 
for $\gamma_1, \gamma_2 \in H^*(X,\mathbb{Z})$. Here we identify any classes $\gamma \in H^*(X, \mathbb{Z})$ with $\gamma \otimes 1 \in \text{QH}(X)$.

In the following, the image of a stable map to $X$ will be called a stable curve in $X$. Given a stable curve $C \subset X$, we let $[C]$ denote the degree in $H_2(X,\mathbb{Z})$ defined by $C$. In particular, we set $[C]=0$ if $C$ is a single point.

\subsection{Distance} For each $\beta \in \Delta$ we set $Z_\beta=G/P_\beta$, where $P_\beta \subset G$ is the unique maximal parabolic subgroup containing $B$ for which $s_\beta \notin W_{P_\beta}$. The group $H_2(Z_\beta, \mathbb{Z})$ is a free rank one $\mathbb{Z}$-module, generated by $[(Z_\beta)_{s_\beta}].$ To simplify notation we identify $H_2(Z_\beta, \mathbb{Z})=\mathbb{Z}\Delta^\vee/\mathbb{Z}\Delta_{P_\beta}^\vee$ with $\mathbb{Z}$, by identifying $[(Z_\beta)_{s_\beta}]$ with 1. Let $\pi_\beta: X \rightarrow Z_\beta$ denote the projection. Any degree $d \in H_2(X,\mathbb{Z})$ is then given by $d=\sum_\beta d_\beta [X_{s_\beta}]$, where $d_\beta=(\pi_\beta)_*(d)$.

Given $u,v \in W^P$ let $\text{dist}_\beta(u,v) \in \mathbb{Z}$ denote the smallest degree of a stable curve in $Z_\beta$ connecting the Schubert varieties $(Z_\beta)^u$ and $(Z_\beta)_v$. This is well defined since the natural numbers are well ordered. Define the \textit{distance} between $X^u$ and $X_v$ to be the class in $H_2(X,\mathbb{Z})$ given by \[ \displaystyle \text{dist}_X(u,v)= \sum_{\beta \in \Delta \backslash \Delta_P} \text{dist}_\beta (u,v) [X_{s_\beta}].\] The following Theorem is from \cite{BCLM:EulerChar}.

\begin{thm} \label{thm:mindeg}
(Buch, Chung, Li, Mihalcea) Let $u,v \in W$ and $d \in H_2(X, \mathbb{Z})$. Then there exists a stable curve of degree $d$ from $X^u$ to $X_v$ if and only if $d \geq \text{dist}_X(u,v)$.
\end{thm}

\subsection{Curve neighborhoods}
Let $d \in H_2(X)$ be an effective degree, and let $\Omega \subset X$ be a closed subvariety. Consider the moduli space of stable maps $\overline{\mathcal{M}}_{0,2}(X,d)$ with evaluation maps $\ev_1, \ev_2.$ The {\it curve neighborhood} of $\Omega$ is the subscheme 
\[ 
\Gamma^X_d(\Omega):=\ev_2(\ev_1^{-1} \Omega) \subset X
\] 
endowed with a reduced scheme structure. The notion was introduced by Buch, Chaput, Mihalcea and Perrin \cite{BCMP:qkfin} to help study the quantum K theory ring of cominuscule Grassmannians.

To compute curve neighborhoods we will use the {\it Hecke product} on the Weyl group $W$. For a simple reflection $s_i$ the product is defined by  \[w \cdot s_i = \left\{
     \begin{array}{lr}
       ws_i &  \mbox{if } \ell(ws_i)>\ell(w); \\
       w &  \mbox{otherwise}
     \end{array} \right. \]
The Hecke product gives $W$ the structure of an associative monoid. Given $u,v \in W$, the product $uv$ is called {\it reduced} if $\ell(uv)=\ell(u)+\ell(v)$. For any parabolic group $P$, the Hecke product determines a left action $W \times W/W_{P} \longrightarrow W/W_{P}$ defined by \[u \cdot (wW_{P})=(u \cdot w) W_P. \] 

Observe from \cite{buch.m:nbhds} that if $\Omega$ is a Schubert variety, then $\Gamma^X_d(\Omega)$ must be a Schubert variety. Let $z_d$ be defined by the condition: $\Gamma^X_d(1.P)=X(z_dW_P)$. The following Proposition describes the curve neighborhood $\Gamma^X_d(X^w)$ as a Schubert variety (see \cite{buch.m:nbhds}).

\begin{prop}
The curve neighborhood of $X^w$ is $\Gamma^X_d(X^w)=X^{w \cdot z_d W_P}$.
\end{prop}

We need the following definition to state Proposition \ref{recursion}.

\begin{defn}
The maximal elements of the set $\{ \beta \in R^+ \backslash R_P^+ : \beta^\vee+ \Delta_P^\vee \leq d\}$ are called {\it maximal roots} of $d$.
\end{defn}

\begin{prop} \label{recursion}
\cite[Corollary 4.12]{buch.m:nbhds} Let $d \in H_2(X)$ be an effective degree and $w \in W^P$. Then if $\alpha \in R^+ \backslash R_P^+$ is a maximal root of $d$, then $s_\alpha \cdot z_{d-\alpha^\vee}W_P=z_dW_P$.  
\end{prop}

Proposition \ref{recursion} is the key ingredient that we use to calculate curve neighborhoods in varying combinatorial models. In particular, we specialize Proposition \ref{recursion} to the $\Gr$ case in Proposition \ref{prop:arecursion}, the $\IG$ case in Proposition \ref{prop:crecursion}, and the submaximal case of $\OG$ in Proposition \ref{prop:brecursion}.

\subsection{Minimum degree from the point of view of curve neighborhoods} \label{subsection:mindegpointofview}We begin this subsection with a lemma that states the correspondence between the  minimum degree and curve neighborhoods.

\begin{lemma} \label{prop:mainprop}
Let $u,v \in W^P$. Let $d$ be the minimal degree such that $X_v \subset \Gamma^X_d(X^u)$ in $X=G/P$. Then $d$ is the smallest power of $q$ in the quantum product $[X^u] \star [X_v].$
\end{lemma}

\begin{proof}
Let $u,v \in W^P$. Let $d$ be the minimal degree such that $X_v \subset \Gamma^X_d(X^u)$ in $X=G/P$. Equivalently, $ v \leq u \cdot z_dW_P$ in the Bruhat order on $W$. By definition $d=\text{dist}_X(u,v)$ is the minimum degree of a curve joining two opposite Schubert varieties $X^u$ and $X_v$. By Theorem \ref{thm:mindeg} the degree $\text{dist}_X(u,v)$ is the smallest power of $q$ in the quantum product $[X^u] \star [X_v]$. The result follows.
\end{proof}

For type $A$, and the submaximal Grassmanians in types $B$ and $C$, we consider four different sets indexing the Schubert subvarieties: the Weyl group, two sets of partitions, and binary strings. For each of these indexing sets, we construct an operator
\[
I\to I: \delta\mapsto \delta^d
\]
(where $I$ represents any of these indexing sets). The next theorem shows the connection between these constructions, curve neighborhoods, and the minimal quantum degree.

\begin{thm} \label{thm:metathm}
Let $X$ be one of the following: the type $A$ (ordinary) Grassmannian, the symplectic Grassmannian, or the odd orthogonal Grassmannian. Let $\delta\in I$ (one of our four indexing sets for Schubert varieties in $X$). Then
\begin{enumerate}
\item $\Gamma_d(X^\delta) = X^{\delta^d}$

\item for $\delta_1,\delta_2\in I$, the smallest power of $q$ in the quantum product $[X^{\delta_1}]\star [X_{\delta_2}]$ is the smallest integer $d$ such that $\delta_2\leq \delta_1^d$ (in the Bruhat order).
\end{enumerate}
\end{thm}

The proof for part (2) follows from part (1) and Lemma \ref{prop:mainprop}. The proof for part (1) follows from the fact $\delta^d = \delta\cdot z_d$. We prove this by explicitly identifying the necessary maximal coroots in each type; for type $A$, we do this in Proposition \ref{prop:arecursion}, for type C in Proposition \ref{prop:crecursion}, and type $B$ in Proposition \ref{prop:brecursion}.

\section{Cominuscule Grassmannians}\label{sect:typeA}
The purpose of this section is to restate Postnikov's \cite{postnikov:qbruhat, postnikov:affine} and Fulton and Woodward's \cite{FW} minimum degree result in the context of curve neighborhoods and Young diagrams for some cominuscule varieties. The cominuscule varieties are the type A Grassmannians $\Gr(k,n)$, the Lagrangian Grassmannians $\IG(n,2n)$, the maximal orthogonal Grassmannians $\OG(n,2n+1)$, the quadric hypersurfaces, the Cayley plane $E_6/P_6$, and the Freudenthal variety $E_7/P_7$. In this paper we prove a Young diagram rule to calculate minimum degrees for the cominuscule Grassmannians, that is, $\Gr(k,n)$, $\IG(n,2n)$, and $\OG(n,2n+1)$. The curve neighborhoods for the cominuscule cases are calculated in \cite[Subsection 3.2, Table 1]{BCMP:eqkt} (See also \cite[Lemma 4.2]{BCMP:qkfin} and \cite{CMP}). We will discuss the type A case in its entirety since the proofs of the submaximal cases in types B and C will follow a similar structure. Then the cominuscule cases in types B and C will be considered in Subsection \ref{subs:BandC}.
\subsection{Grassmannian}
Let $\Gr:=\Gr(k,n)$ be the Grassmannian of $k$-dimensional subspaces in $\mathbb{C}^n$. The notation in this subsection is strictly for type A. We begin by defining three indexing sets---permutations, partitions, and 01-words. This is followed by Lemma \ref{lem:TypeAbij} that states the bijections between the index sets.

\begin{defn} We define the set of partitions $\mathcal{P}(k,n)$. \[ \mathcal{P}(k,n)= \{( \lambda_1 \geq \cdots \geq \lambda_k): \mbox{ where } n-k \geq \lambda_1 \mbox{ and } \lambda_k \geq 0 \}.\]
\end{defn}

Let $P$ be the maximal parabolic obtained by excluding the reflection $s_k$. Then the minimal length representatives $W^{P}$ have the form \[ (u(1)<u(2)<\cdots<u(k)\mid u(k+1)< \cdots < u(n)) \in W^{P} \subset S_n.\] Since the last $n-k$ labels are determined from the first, we will identify an element in $W^{P}$ with the sequence $(u(1)<u(2)<\cdots<u(k))$.

\begin{defn}
Let $W_{01}$ denote the set of 01-words with $k$ ones and $n-k$ zeros.
\end{defn}

\begin{lemma}\label{lem:TypeAbij} \leavevmode
\begin{enumerate}
\item There is a natural bijection between the minimal representatives of $W^{P}$ and the set $W_{01}$ of 01-words with $n-k$ zeros and $k$ ones. It is given by the following: The element $u \in W^{P}$ corresponds to the word where the ones appear in the $u(1),u(2), \cdots, u(k)$ positions reading left to right.
\item There is a natural bijection between $W_{01}$ and $\mathcal{P}(k,n)$. It is given as follows: Produce the partition $\lambda \in \mathcal{P}(k,n)$ by reading the 01-word from left to right. Starting at the top-right corner, proceed south-west by moving left for each 0, and down for each 1.
\end{enumerate}
\end{lemma}

In type A, there is a correspondence between partitions in the $k\times (n-k)$ rectangle and Grassmannian permutations with descent at $k$. Given such a partition $\lambda$, define
\[
u(\lambda)[i] = i + \lambda(k+1-i)
\]
for $i=1,2,\ldots,k$. For reasons that will become apparent in type C, we would like permutations to be indexed homologically with partitions indexed cohomologically (i.e. $|\lambda| = \text{codim}(X)$ and $\ell(u) = \text{dim}(X)$ for the same Schubert variety $X$). For this reason, we would rather use a correspondence between $\lambda$ and $u_{\lambda} := u(\lambda^{\vee})$. 

The following definitions capture the combinatorics of curve neighborhoods in terms of partitions. The first and third definitions correspond to the line neighborhood of the corresponding Schubert variety, and the second and fourth correspond to the degree $d$ curve neighborhood of the corresponding Schubert variety.

\begin{defn}\label{def:typeAcombforcurve}\leavevmode
\begin{enumerate}
\item Let $\lambda \in \mathcal{P}(k,n)$. Then define $\lambda^1=(\lambda_2-1 \geq \cdots \geq \lambda_k-1 \geq 0) \in \mathcal{P}(k,n)$ where $-1$'s are replaced by $0$.

\item Define $\lambda^d=( \cdots ((\lambda^1)^1) \cdots)^1$ ($d$-times).

\item Let $u=(u(1)<u(2)<\cdots<u(k)) \in W^{P}$. Define $u^d$ in the following way.

\begin{enumerate}
\item For the case $d=1$.
\begin{enumerate}
\item If $u(k)<n$ then define $u^1:=\left(u(2)<u(3)<\cdots<u(k)< n\right)$;
\item If  \[ u=\left(u(1)<u(2)<\cdots <u(k-j_0-1)<\overbrace{n-j_0}^{u(k-j_0)}<\cdots<\overbrace{n-1}^{u(k-1)}< \overbrace{n}^{u(k)}\right)\] where $u(k-i)=n-i \mbox{ for } 0 \leq i \leq j_0$ and $u(k-j_0-1)+1<n-j_0$ then define \[u^1=\left(u(2)<\cdots <\overbrace{u(k-j_0-1)}^{u^1(k-j_0-2)}< \overbrace{n-j_0-1}^{u^1(k-j_0-1)}<\overbrace{n-j_0}^{u^1(n-j_0)}<\cdots<n-1< n\right). \]
\end{enumerate}

\item If $d>1$ then define $u^d$ recursively by $(u^{d-1})^1$.
\end{enumerate}

\item Let $\gamma \in W_{01}$ and define $\gamma^1$ by: begin by deleting the first 1 occurring in $\gamma$, next replace each $10$-string with a $01$-string, and finally append 1 to the end of the resulting word.

\item Define $\gamma^d$ recursively by $\gamma^d = (\gamma^{d-1})^1$.
\end{enumerate}
\end{defn}

At this point, we have three combinatorial objects that are related to Schubert varieties: permutations, partitions, and 01-words. We also have bijections between these related objects. We now show that these bijections are actually ``dictionaries'' which also preserve the combinatorics of curve neighborhoods.

\begin{prop} \label{prop:TypeAcnbdcorres}
 Suppose $\lambda \in \mathcal{P}(k,n)$, $u \in W^{P}$, and $\gamma \in W_{01}$ are in bijection with one another. Then $\lambda^d$, $u^d$, and $\gamma^d$ are in bijection with one another.
\end{prop}

\begin{proof}
Observe that $\lambda^1$ corresponds to $\gamma^1$ by definition. So, it suffices to show that $u^1$ corresponds to $\gamma^1$. Using the correspondence between permutations and 01-words, it is clear that $u^1$ corresponds to a 01-word obtained from $\gamma(u)$ by interchanging the first 1 and the last 0 occurring in $\gamma(u)$. Thus, the problem is further simplified to: show that $\gamma^1$ is  constructed by interchanging the first 1 and the last 0 in $\gamma$. 

Write
\[
\gamma = \textbf{1}.0.u.1.\textbf{0}
\]
where $\textbf{0}$ is a string of all zeros, and similar for $\textbf{1}$. The procedure defining $\gamma^1$ produces
\[
\gamma^1 = 1.\textbf{1}.u.0.\textbf{0} = \textbf{1}.1.u.0.\textbf{0}
\]
(since replacing all $10$-strings with $01$-strings results in $u.0$ being replaced by $0.u$).
\end{proof}

\begin{example}
Let $n=12$ and $k=4$. Consider $u=(3<4<6<8) \in W^{P}.$ Then $u^1=(4<6<8<12)$, $\gamma=001101010000$, and $\gamma^1=000101010001$, and \[ \lambda=\yng(6,6,5,4) \text{ and } \lambda^1=\yng(5,4,3,0). \]
\end{example}

We next specialize Proposition \ref{recursion} to type A. Let \[ \{t_i-t_j :1 \leq i<j \leq n \}\]  be the set of positive roots where \[ \{t_i-t_{i+1}: 1 \leq i \leq n-1 \} \] is the set of simple roots. 
\begin{prop} \label{prop:arecursion} We have the following statements.
\begin{enumerate}
\item  The root is $t_1-t_n$ is a maximal root of any degree $d>0$ and \[ z_1=s_{t_1-t_n}=s_1s_2 \ldots s_{n-2}s_{n-1}s_{n-2} \ldots s_2 s_1.\] 
\item Let $u \in W^{P}$ and $d>0$. Then
\begin{enumerate}
\item $z_d=z_1 \cdot \ldots \cdot z_1$ ($d$-times);
\item $u^1$ is the minimal length coset representative of $u \cdot z_1W_P$;
\item $u^d$ is the minimal length coset representative of $u^{d-1} \cdot z_1 W_P$;
\item $u^d$ is the minimal length coset representative of $u \cdot z_d W_P$.
\end{enumerate}
\end{enumerate}
\end{prop}

We will give an example before we prove Proposition \ref{prop:arecursion}.
\begin{example}
Consider $u=2413$ in the Weyl group of $\Gr(2,4)$. Then by the definition of the Hecke product we have that
\[
u \cdot z_1 = 2413\cdot s_1s_2s_3s_2s_1= 4213\cdot s_2s_3s_2s_1= 4213\cdot s_3s_2s_1= 4231\cdot s_2s_1= 4321\cdot s_1=4321.
\]
Notice that $u^1=(3<4|1<2)$ is the minimum length coset representative of $u \cdot z_1 W_P$.
\end{example}

\begin{proof}
First observe the following standard facts. The highest root is $t_1-t_n$ and its coroot is $t_1-t_n$. Then, \[t_1-t_n=(t_1-t_2)+\cdots+(t_{n-1}-t_n). \] In particular, each coefficient is one and $\alpha:=t_1-t_n \in R^+\backslash R^+_P$ is a maximal root for any $d>0$. By Proposition \ref{recursion}, $s_{\alpha} \cdot z_{d-\alpha^\vee}W_P=z_dW_P$, $z_1=s_{\alpha}$, and $z_d=z_1 \cdot \ldots \cdot z_1$ ($d$-times). A direct calculation shows that $u^1$ is the minimal length coset representative of $u \cdot z_1W_P$. Finally, part 3b of Definition \ref{def:typeAcombforcurve} indicates that $u^d=(u^{d-1})^1$. So, $u^d$ is the minimal length coset representative of $u^{d-1} \cdot z_1 W_P$. The result follows.
\end{proof}

We can now state the Young diagram formula for type A.
\begin{thm} \label{thm:visthma}
For any pair of Schubert classes $[X^\lambda], [X_\mu] \in \QH^*(\Gr(k,n))$ where $\lambda, \mu \in \mathcal{P}(k,n)$, the smallest degree $d$ such that $q^d$ appears in  $[X^\lambda] \star [X_\mu]$ with nonzero coefficient is the smallest integer $d$ such that $\lambda^d \subset \mu$.
\end{thm}

\begin{example} \label{ex:A}
For $\Gr(5,16)$ consider $\lambda=(11 \geq 11 \geq 11 \geq 4 \geq 4) \in \mathcal{P}(5,16)$ and $\mu=(7,7) \in  \mathcal{P}(5,16).$ Then we have that 

\[ \lambda=\yng(11,11,11,4,4) \mbox{ and } \mu=\yng(7,7)\]
\begin{align*}
\lambda^1 &= \yng(10,10,3,3)\\
\lambda^2 &= \yng(9,2,2)\\
\lambda^3 &= \yng(1,1).
\end{align*}
Thus, 3 is the smallest degree $d$ such that $q^d$ appears in  $[X^\lambda] \star [X_\mu]$ with nonzero coefficient.
\end{example}

\subsection{The Lagrangian Grassmannian and the maximal odd orthogonal Grassmannian} \label{subs:BandC}

We will first define the isotropic Grassmannians in types B and C. In this subsection we will only consider the cominuscule case for each type.  The remainder of the paper will be dedicated to the submaximal cases in types B and C.

Fix the vector space $\mathbb{C}^{2n}$ with a non-degenerate skew-symmetric bilinear form $(\cdot\,,\cdot)$, and fix a non-negative integer $k$. The symplectic Grassmannian $\IG(k,2n)$ parameterizes $k$-dimensional isotropic subspaces of $\mathbb{C}^{2n}$. Similarly, consider the vector space $\mathbb{C}^{2n+1}$ with a non-degenerate symmetric bilinear form. The odd orthogonal Grassmannian $\OG(k,2n+1)$ parameterizes $k$-dimensional isotropic subspaces of $\mathbb{C}^{2n+1}$.

The notation used in this section is only for the cominuscule cases. Let \[X \in \{ \IG(n,2n), \OG(n,2n+1) \}.\] The Schubert varieties are indexed by the set of strict partitions \[\Lambda:=\{(\lambda_1 \geq \cdots \geq \lambda_n): n \geq \lambda_1,\lambda_k \geq 0, \lambda_i>0 \implies \lambda_i>\lambda_{i+1} \}. \]

The curve neighborhoods are calculated in \cite[Subsection 3.2, Table 1]{BCMP:eqkt}. We state the curve neighborhoods in the next lemma.
\begin{lemma}The following results hold:
\begin{enumerate}
\item Let $X^\lambda_C \subset \IG(n,2n)$ be a Schubert variety for some $\lambda_C \in \Lambda$. Then \[\Gamma_d^\IG(X^{\lambda_C})=X^{\lambda_C^d}\] where \[ \lambda_C^d=(\lambda_{d+1}\geq \lambda_{d+2} \geq \cdots \geq \lambda_k \geq 0 \geq \cdots 0).\]
\item Let $X^{\lambda_B} \subset \OG(n,2n+1)$ be a Schubert variety for some $\lambda_B \in \Lambda$. Then \[\Gamma_d^\OG(X^{\lambda_B})=X^{\lambda_B^d}\] where \[ \lambda_B^d=(\lambda_{2d+1}\geq \lambda_{2d+2} \geq \cdots \geq \lambda_k \geq 0 \geq \cdots 0).\]
\end{enumerate}
\end{lemma}

We can now state the Young diagram formula for $\IG(n,2n)$ and $\OG(n,2n+1)$.
\begin{thm} \label{thm:visthmaximal}
Let $X \in \{ \IG(n,2n), \OG(n,2n+1)\}$. For any pair of Schubert classes $[X^\lambda], [X_\mu] \in \QH^*(X)$ where $\lambda, \mu \in \Lambda$, the smallest degree $d$ such that $q^d$ appears in  $[X^\lambda] \star [X_\mu]$ with nonzero coefficient is the smallest integer $d$ such that $\lambda^d \subset \mu$.
\end{thm}

\begin{example}
Consider the case $\IG(5,10)$. Let $\lambda=(4,3,2,1) \in \Lambda$ and $\mu=(3,1) \in \Lambda$. Then we have that \[\lambda=\yng(4,3,2,1) \mbox{and } \mu=\yng(3,1). \]  We have that \[ \lambda^1=\yng(3,2,1) \mbox{and } \lambda^2=\yng(2,1).\] So the minimum degree $d$ that appears in the product $[X^\lambda] \star [X_\mu]$ is $2$. 
\end{example}
\begin{example}
Consider $\OG(5,11)$. Let $\lambda=(4,3,2,1) \in \Lambda$ and $\mu=(3,1) \in \Lambda$. Then we have that \[ \lambda^1=\yng(2,1). \] So the minimum degree $d$ that appears in the product $[X^\lambda] \star [X_\mu]$ is $1$.
\end{example}

\section{Indexing sets for submaximal Grassmannians in types B and C}\label{sect:combo}



We begin by defining four indexing sets---permutations, two kinds of partitions, and 01-words. The indexing sets for the Schubert varieties of $\IG(k,2n)$ and $\OG(k,2n+1)$ are the same. However, curve neighborhood calculations depend on being in either type C or type B. When ambiguity arises, subscripts on elements of indexing sets will be used to indicate whether they are to be considered in type B or type C for this section and Section \ref{sect:curvenbhds}. Finally, Lemma \ref{lem:TypeBCbij} states the bijections between the index sets.



The associated Weyl group $W$ is the hyperoctahedral group consisting of {\em signed permutations}, i.e. permutations $u$ of the elements $\{1, \cdots, n,\overline{n},\cdots,\overline{1}\}$ satisfying $u(\overline{i})=\overline{u(i)}$ for all $i=1,\ldots,n$. There is a natural ordering 
\[ 
1 < 2 < \ldots < n < \overline{n} < \ldots < \overline{1}.
\]
We define $\bar{i}=2n+1-i$ and $|i|=\min \{i,2n+1-i \}$.

Let $P_k$ be the maximal parabolic obtained by excluding the reflection $s_k$. Then the minimal length representatives $ W^{P}$ have the form 
\[ 
(u(1)<u(2)<\cdots<u(k)\mid u(k+1)< \cdots < u(n) \le n) \in W^{P}
\] 
if $k < n$ and $(u(1)<u(2)<\cdots<u(n))$ if $k= n$. Since the last $n-k$ labels are determined from the first, we will identify an element in $W^{P}$ with the sequence \[ (u(1)<u(2)<\cdots<u(k)).\]

The next two definitions include two kinds of partitions. The first are $(n-k)$-strict partitions. The second is a particular subset of $\mathcal{P}(k,2n)$.

\begin{defn}
Let
\[ 
\Lambda = \{(\lambda_1 \geq \ldots \geq \lambda_k): 2n-k \geq \lambda_1, \lambda_k \geq 0, \mbox{ and if } \lambda_j>n-k \mbox{ then } \lambda_{j+1}<\lambda_j   \}
\] 
denote the set of $(n-k)$-strict partitions.\footnote{ It should be noted that while the Bruhat order is compatible with partition inclusion in type A, the Bruhat orders in types B and C are not computable by using partition inclusion with $k$-strict partitions. Indeed, $(2n-k) \leq (1,1, \cdots, 1)$ in the Bruhat order for $k<n$.} Let $\ell(\lambda)=\max \{j : \lambda_j>0\}$ and $\ell_1(\lambda)=\max \{j : \lambda_j>1\}$.
\end{defn}

The next definition states those partitions and a relevant definition. We use the point of view that the 01-words for $\IG(k,2n)$ and $\OG(k,2n+1)$ are also 01-words for $\Gr(k,2n)$.

\begin{defn}
Let $\lambda \in \mathcal{P}(k,n)$ be a partition. This partition's boundary consists of $n$ steps moving either left or down in the south-west direction. Let $D(\lambda)(i)=0$ if the $i$th step is left and $D(\lambda)(i)=1$ if the $i$th step is down.
\end{defn}

\begin{defn} \label{def:partbruh} 
Define \[ \mathcal{P}'(k,2n)=\{\lambda \in  \mathcal{P}(k,2n): \mbox{ if } D(\lambda)(i)=D(\lambda)(\bar{i}) \mbox{ then } D(\lambda)(i)=0 \}.\] Define $W_{01}$ to be the set of 01-words that correspond to partitions in $\lambda$.
\end{defn}


\begin{example}
The partition $(5,2,1) \in \mathcal{P}'(3,8)$. The corresponding 01-word is 10001010. Pictorially, \[\yng(5,2,1) \in \mathcal{P}'(3,8).  \] On the contrary, the partition $(5,5,1) \notin \mathcal{P}'(3,8)$. The corresponding 01-word is 11000010. Pictorially, \[\yng(5,5,1) \notin \mathcal{P}'(3,8).  \] 
\end{example}

The next lemma is a dictionary between the different combinatorial models. Part (1) is stated and proved in \cite{BKT,BKT2}.
\begin{lemma} \label{lem:TypeBCbij} \leavevmode
\begin{enumerate}
\item There is a bijection between $\Lambda$, the collection of $(n-k)$-strict partitions, and the set $W^{P}$ of minimal length representatives given by:
\begin{align*}
&\lambda \mapsto w& \mbox{ where } w(j)=2n+1-k-\lambda_j+\#\{i<j: \lambda_i+\lambda_j \leq 2(n-k)+j-i \},\\
&w \mapsto \lambda& \mbox{ where } \lambda_j=2n+1-k-w(j)+\# \{i<j: w(i)+w(j)>2n+1 \}.
\end{align*}

\item There is a natural bijection between the minimal length representatives of $ W^{P}$ and the set $ W_{01}$. It is given by the following: The element $w \in W^{P}$ corresponds to the word where the ones appear in the $w(1),w(2), \cdots, w(k)$ positions reading left to right.

\item The set of $W_{01}$ and $\mathcal{P}'(k,2n)$ have a canonical bijection. The set \[\mathcal{P}'(k,2n) \subset  \mathcal{P}(k,2n)\] is the set of partitions where $\lambda \in \mathcal{P}'(k,2n)$ is found by reading the 01-word from left to right. Starting at the top-right corner, proceed south-west by moving left for each 0, and down for each 1.
\end{enumerate}
\end{lemma}

\section{Curve neighborhood combinatorics}\label{sect:curvenbhds}

As in Section \ref{sect:typeA}, we provide descriptions of the combinatorics associated to curve neighborhoods. These definitions are spread out among the different types of combinatorics. Definitions \ref{defn:typeCpartitions}, \ref{defn:typeBpartitions1} and \ref{defn:typeBpartitionsd} describe the $k$-strict partition point of view in type C, type B for $d=1$ and $k>1$, and type B for the other cases, respectively. Definitions \ref{def:Cpart2} and \ref{def:Bpart2} describe the second set of partitions point of view in types C and B, respectively. Definitions \ref{defn:typeCperm} and \ref{defn:typeBperm} describe the permutation point of view in types C and B, respectively. Finally, Definition \ref{defn:binarywords} describes the 01-word perspective (in both types).

\begin{defn}\label{defn:typeCpartitions}
Let $\lambda_C \in  \Lambda$. Define $\lambda_C^d$ in the following way:
\begin{enumerate}
\item If $\lambda_1+\lambda_j > 2(n-k)+j-1$ for all $2 \leq j \leq k$ then define \[ \lambda_C^1=(\lambda_2 \geq \lambda_3 \geq \cdots \geq \lambda_k \geq 0) \in \Lambda; \] 
\item Otherwise, find the smallest $j$ such that $\lambda_1+\lambda_j \leq 2(n-k)+j-1$. Define \[ \lambda_C^1=(\lambda_2 \geq \lambda_3 \geq \cdots \geq \lambda_{j-1} \geq \lambda_j-1 \geq \cdots \geq \lambda_{k}-1 \geq 0) \in  \Lambda \] where $-1$'s are replaced by 0;
\item Define $\lambda_C^d=( \lambda_C^{d-1})^1$ for $d>1$.
\end{enumerate}
\end{defn}
\begin{example}
Consider the case $n=8$ and $k=5$. Let $u_C=(7<\bar 8<\bar 5<\bar 4 < \bar 2 ) \in W^P$. Then $u_C^1=(\bar 8<\bar 5< \bar 4< \bar 2< \bar 1)$ and \[ \lambda_C=\yng(5,3,2,2,1) \text{ and } \lambda_C^1=\yng(3,1,1). \]
\end{example}

For curve neighborhood combinatorics in terms of $(n-k)$-strict partitions for $\OG(k,2n+1)$, we first define the $d=1$ case for $2 \leq k \leq n-1$.

\begin{defn} \label{defn:typeBpartitions1}Consider the case $2 \leq k \leq n-1$. Let $\lambda_B \in  \Lambda$. Define $\lambda_B^1$ in the following way:

\begin{enumerate}
\item If $\lambda_1-\ell(\lambda_B) > 2(n-k)$ then $\lambda_B^1$ is defined to be as follows.
\begin{enumerate}
\item If $\lambda_1+\lambda_j > 2(n-k)+j-1$ for all $2 \leq j \leq k$ then define \[ \lambda_B^1=(\lambda_2 \geq \lambda_3 \geq \cdots \geq \lambda_{\ell(\lambda_B)} \geq \overbrace{1 \geq \cdots \geq 1}^{1+2k+\lambda_1-2n-\ell(\lambda_B)} \geq \overbrace{0 \geq \cdots \geq 0}^{2n-k-\lambda_1}) \in \Lambda; \] 
\item Otherwise, find the smallest $j$ such that $\lambda_1+\lambda_j \leq 2(n-k)+j-1$. 
\begin{enumerate}
\item If $\ell_1(\lambda_B) \geq j$ then \[ \lambda_B^1=(\lambda_2 \geq \lambda_3 \geq \cdots \geq \lambda_{j-1} \geq \lambda_j-1 \geq \cdots \geq \lambda_{\ell_1(\lambda_B)}-1 \geq \overbrace{1 \geq \cdots \geq 1}^{1+2k+\lambda_1-2n-\ell_1(\lambda_B)} \geq \overbrace{0 \geq \cdots \geq 0}^{2n-k-\lambda_1}) \in  \Lambda; \]
\item If $\ell_1(\lambda_B) < j$ then \[ \lambda_B^1=(\lambda_2 \geq \lambda_3 \geq \cdots  \geq \lambda_{\ell_1(\lambda_B)} \geq  \overbrace{1 \geq \cdots \geq 1}^{1+2k+\lambda_1-2n-\ell_1(\lambda_B)} \geq \overbrace{0 \geq \cdots \geq 0}^{2n-k-\lambda_1}) \in  \Lambda; \]
\end{enumerate}
\end{enumerate}
\item Otherwise, define $\lambda_B^1:=\lambda_C^1$.
\end{enumerate}
\end{defn}

Next we define the curve neighborhoods combinatorics in terms of $(n-k)$-strict partitions for the $\OG(1,2n+1)$ case and for the case when $d>1$ for $\OG(k,2n+1)$.

\begin{defn}\label{defn:typeBpartitionsd}
Let $\lambda_B \in  \Lambda$. Define $\lambda_B^d$ in the following way:
\begin{enumerate}
\item Consider the case $k=1$. If $\lambda_B=(2n-k)$ then define $\lambda_B^1=(1)$. If $\lambda_B \neq (2n-k)$ or $d>1$ then $\lambda_B^d=(0)$.
\item Consider the case $2 \leq k \leq n-1$.
\begin{enumerate}
\item If $d>1$  and even then define $\lambda_B^d:=\lambda^d_C$.
\item If $d>1$ and odd then define $\lambda_B^d:=(\lambda_B^{d-1})^1$.

\end{enumerate}
\end{enumerate}
\end{defn}

\begin{example}
Consider the case $n=8$ and $k=5$. Let $u_B=(1<5<\bar 8<\bar 4 <\bar 2 ) \in W^P$. Then $u_B^1=(5<\bar 8< \bar 4< \bar 3<\bar 2)$ and \[ \lambda_B=\yng(11,7,3,1,0) \text{ and } \lambda_B^1=\yng(7,3,1,1,1). \]
In this example, $\lambda_1-\ell(\lambda_B)>2(n-k)$ and $2n-k-\lambda_1=0$.
\end{example}

Next we will define the curve neighborhood combinatorics in terms of partitions in $\mathcal{P}'(k,2n)$. First for type C.

\begin{defn} \label{def:Cpart2}
Let $\mu_C \in \mathcal{P}'(k,2n)$. Then  
\begin{enumerate}
\item define $\mu_C^1=\mu_A^1$. That is, \[\mu_C^1:=(\mu_2-1, \mu_3-1, \cdots, \mu_k-1,0) \] where $-1$'s are replaced by 0's.
\item define $\mu_C^d=(\mu_C^{d-1})^1$ for $d>1$.
\end{enumerate}
\end{defn}

For type B combinatorics, first recall that the partition $\lambda \in \mathcal{P}(k,n)$ has a boundary consisting of $n$ steps moving either left or down in the south-west direction. We define $D(\lambda)(i)$ to be $D(\lambda)(i)=0$ if the $i$th step is left and $D(\lambda)(i)=1$ if the $i$th step is down.

\begin{defn}
Let $\lambda \in \mathcal{P}(k,n)$. Let $\lambda^t$ denote the transpose of $\lambda$. We say that $\lambda$ is {\it $m$-wingtip symmetric}\footnote{To the best of the authors' knowledge this property is unnamed in the literature. The part of the name ``wingtip" came from considering Young diagrams as a fixed wing aircraft.} if $m$ is the largest nonnegative integer such that \[ D(\lambda)(i)=D(\lambda^t)(i) \] for all $i$, $1 \leq i \leq m$.
\end{defn}

\begin{example}
Consider $\mu=(10 \geq 8 \geq 3 \geq 1 \geq 0) \in \mathcal{P}'(5,2 \cdot 8) \subset \mathcal{P}(5,2 \cdot 8)$. The partition $\mu$ is $3$-wingtip symmetric. It corresponds to the 01-word $0100100000100101$. Here the first and last character, the second and second to last, and the third and third to last character have opposite parity. Pictorially,

\[ \mu=\yng(10,8,3,1,0).   \] 

The first 3 bottom left hand boundary edges yield the partition $\gamma=\yng(1,0)$ and the first 3 top right hand boundary edges yields the partition $\gamma^t$.
\end{example}

\begin{defn} \label{def:Bpart2}
Let $\mu_B \in \mathcal{P}'(k,2n)$ be $m$-wingtip symmetric. Then define $\mu_B^d$ in the following way.
\begin{enumerate}
\item For the case $k=1$. If $\lambda_B=(2n-k)$ then define $\lambda_B^1=(1)$. If $\lambda_B \neq (2n-k)$ or $d>1$ then $\lambda_B^d=(0)$.
\item For the case that  $1<k<2n$ we have the following subcases.
\begin{enumerate} 
\item If $d=1$ then $\mu_B^1$ is defined by
\begin{enumerate}
\item If $D(\mu_B)(\overline m )= 1$ and at the end of the $i$th row then \[\mu_B^1:=(\mu_2-1 \geq \cdots \geq \mu_{i-1}-1, \mu_i,\mu_i, \mu_{i+1}, \cdots, \mu_k ); \] 
\item If $D(\mu_B)(\overline m )= 0$, at the bottom of the $i$th row, and in the $j$th column then  \[\mu_B^1:=(\mu_2-1 \geq \cdots \geq \mu_{i}-1, j, \mu_{i+1}, \cdots, \mu_k ). \]
\end{enumerate}
\item If $d>1$ and even define $\mu_B^d:=\mu_C^d;$
\item If $d>1$ and odd define $\mu_B^d=(\mu_B^{d-1})^1$.
\end{enumerate}
\end{enumerate}
\end{defn}
\begin{example}
Consider $\mu=(10 \geq 8 \geq 3 \geq 1 \geq 0) \in \mathcal{P}'(5,2 \cdot 8) \subset \mathcal{P}(5,2 \cdot 8)$. This partition is 3-wingtip symmetric. Pictorially, 
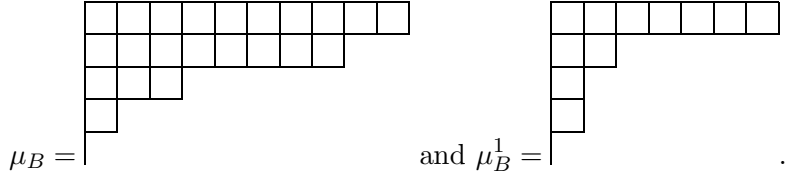
\[ \mu_B=\yng(10,8,3,1,0) \mbox{ and } \mu^1_B=\yng(7,2,1,1,0).\]
\end{example}

\begin{defn}\label{defn:typeCperm} Let $u_C=(u(1)<u(2)<\cdots<u(k)) \in W^{P}$. Define $u_C^d$ in the following way.
\begin{enumerate}
\item For the case case $d=1$.
\begin{enumerate}
\item If $u(k)< \bar{1}$ then define $u_C^1:=\left(u(2)<u(3)<\cdots<u(k)< \bar{1} \right)$;
\item If  \[ u_C=\left(u(1)<u(2)<\cdots <u(k-j_0-1)<\overbrace{\overline{j_0+1}}^{u(k-j_0)}<\cdots<\overbrace{\overline{2}}^{u(k-1)}< \overbrace{\bar{1}}^{u(k)}\right)\] where $u(k-i)=\overline{i+1} \mbox{ for } 0 \leq i \leq j_0$ and $u(k-j_0-1)+1<\overline{j_0+1}$ then define \[u_C^1=\left(u(2)<\cdots <\overbrace{u(k-j_0-1)}^{u^1(k-j_0-2)}< \overbrace{\overline{j_0}}^{u^1(k-j_0-1)}<\overbrace{\overline{j_0+1}}^{u^1(n-j_0)}<\cdots<\bar{2}< \bar{1} \right). \]
\end{enumerate}
\item If $d>1$ then define $u^d$ recursively by $(u^{d-1})^1$.
\end{enumerate}
\end{defn}

\begin{defn}\label{defn:typeBperm} Let $u_B \in W^{P}$.
\begin{enumerate}
\item Consider the case $k=1$. If $u_B=(1)$ then define $u_B^1=(\bar{2})$. If $u_B \neq (1)$ or if $d>1$ then define $u_B^d=(\bar{1})$.
\item Consider the case $2 \leq k \leq n-1$.
\begin{enumerate}
\item For the case $d=1$. Let \[ u_B=\left(u(1)<u(2)<\cdots <u(k)\right).\] Let $J=\min \left( \{1,2,\cdots, n \} \backslash \{ |u(1)|, \cdots, |u(k)|\} \right)$ and $j_0= \min \left( \{ j: u(j) >\bar{J} \} \cup \{ k+1\} \right)$. 
\begin{enumerate}
\item If $j_0=k+1$ then define \[u_B^1=\left(u(2)<u(3)<\cdots<u(k)<\bar{J} \right). \]
\item If $j_0 \leq k$ then define \[u_B^1=\left(u(2)<u(3)<\cdots<u(j_0-1)<\bar{J}<u(j_0)<\cdots<u(k) \right). \]
\end{enumerate}
 \item If $d>1$ and even then define $u_B^d:=u_C^{d}$.
 \item If $d>1$ and odd then define $u_B^d:=(u_B^{d-1})^1$.
 \end{enumerate}
 \end{enumerate}
\end{defn}

\begin{defn}\label{defn:binarywords}\leavevmode
\begin{enumerate}
 \item Let $\gamma_C \in W_{01}$ and let $\gamma_C \mapsto u_C \in W^{P}$. Define $\gamma_C^d \in W_{01}$ to be the 01-word in bijection with $u_C^d$.
 \item Let $\gamma_B \in W_{01}$ and let $\gamma_B \mapsto u_B \in W^{P}$. Define $\gamma_B^d \in W_{01}$ to be the 01-word in bijection with $u_B^d$.
 \end{enumerate}
 \end{defn}




\subsection{Dictionary correspondence for curve neighborhood combinatorics}

The proofs of Propositions \ref{prop:C} and \ref{prop:B} are reserved until Section \ref{sect:techproof}.

\begin{prop} \label{prop:C}
 Suppose $\lambda_C \in \Lambda_C$, $\mu_C \in \mathcal{P}'(k,2n)$, $u_C \in W^{P}$, and $\gamma_C \in W_{01}$ are in bijection with one another and $d \geq 0$. Then $\lambda_C^d$, $\mu_C^d$, $u_C^d$, and $\gamma_C^d$ are in bijection with one another.
\end{prop}

\begin{example}
Let $n=6$ and $k=4$. Consider $u_C=(3<4<6<\bar{5}) \in  W^{P}.$ Then $u_C^1=(4<6<\bar{5}<\bar{1})$, $\gamma_C=001101010000$, and $\gamma_C^1=000101010001$,  \[ \lambda_C=\yng(6,5,3,2) \text{ and } \lambda_C^1=\yng(5,3,2,0), \] and \[ \mu_C=\yng(6,6,5,4) \text{ and } \mu_C^1=\yng(5,4,3,0), \]
\end{example}

\begin{prop} \label{prop:B}
 Let $\lambda_B \in \Lambda_B$, $\mu_C \in \mathcal{P}'(k,2n)$, $u_B \in W^{P}$, and $\gamma_B \in W_{01}$ be in bijection with one another and $d \geq 0$. Then $\lambda_B^d$, $\mu_B^d$, $u_B^d$, and $\gamma_B^d$ are in bijection with one another.
\end{prop}

\subsection{Curve Neighborhoods}\label{sect:qdeg}

In this subsection we calculate curve neighborhoods for submaximal isotropic Grassmannians in both types B and C. We begin with type C.

\subsection{Type C}  We first restate Proposition \ref{recursion} for type C. Let \[ \{t_i \pm t_j :1 \leq i \leq j \leq n \}\] be the set of positive roots where \[ \{t_i-t_{i+1}: 1 \leq i \leq n-1 \} \cup \{2t_n \}\] is the set of simple roots. 

\begin{prop} \label{prop:crecursion} We restate the recursion for type C.
\begin{enumerate}
\item The root $2t_1$ is a maximal root of any $d>0$ and \[z_1=s_{2t_1}=s_1s_2 \ldots s_n \ldots s_2 s_1.\] 
\item Let $u_C \in W^{P}$ and $d>0$. Then
\begin{enumerate}
\item $z_d=z_1 \cdot \ldots \cdot z_1$ ($d$-times);
\item $u_C^1$ is the minimal length coset representative of $u_C \cdot z_1W_P$;
\item $u_C^d$ is the minimal length coset representative of $u_C^{d-1} \cdot z_1 W_P$;
\item $u_C^d$ is the minimal length coset representative of $u_C \cdot z_d W_P$.
\end{enumerate}
\end{enumerate}
\end{prop}

We will provide an example before we prove Proposition \ref{prop:crecursion}.
\begin{example}
Consider $u_C=12\bar435$ in the Weyl group of $\IG(3,10)$. Then the Hecke product $u_C \cdot z_1$ is
\begin{align*}
u_C \cdot z_1&= 12\bar435 \cdot s_1s_2s_3s_4s_5s_4s_3s_2s_1=  21\bar435 \cdot s_2s_3s_4s_5s_4s_3s_2s_1=2\bar4135 \cdot s_3s_4s_5s_4s_3s_2s_1\\
&= 2\bar4315 \cdot s_4s_5s_4s_3s_2s_1= 2\bar4351 \cdot s_5s_4s_3s_2s_1= 2\bar435\bar1 \cdot s_4s_3s_2s_1= 2\bar43\bar15 \cdot s_3s_2s_1\\
&= 2\bar4\bar135 \cdot s_2s_1=  2\bar1\bar435  \cdot s_1=\bar12\bar435.
\end{align*}
Finally notice that $u_C^1=(2<\bar4<\bar1|3<5)$ is the minimum length coset representative of $u_C \cdot z_1W_P$.
\end{example}

\begin{proof}
First observe the following standard facts. The highest root is $2t_1$ and its coroot expands as a sum of simple coroots as follows: \[t_1=(t_1-t_2)+(t_2-t_3)+\cdots+(t_{n-1}-t_n)+t_n. \]  In particular, each coefficient is one and $\alpha:=2t_1 \in R^+\backslash R^+_P$ is a maximal root for any $d>0$. By Proposition \ref{recursion}, $s_{\alpha} \cdot z_{d-\alpha^\vee}=z_dW_P$, $z_1=s_{\alpha}$, and $z_d=z_1 \cdot \ldots \cdot z_1$ ($d$-times). A direct calculation shows that $u_C^1$ is the minimal length coset representative of $u_C \cdot z_1W_P$. Finally, part 2 of Definition \ref{defn:typeCperm} indicates that $u_C^d=(u_C^{d-1})^1$. So, $u_C^d$ is the minimal length coset representative of $u_C^{d-1} \cdot z_1 W_P$. The result follows.
\end{proof}

The next theorem describes the curve neighborhoods of Schubert varieties in $\IG(k,2n)$.
\begin{thm} \label{Thm:Ccrv}
Let $\lambda \in \Lambda, \mu \in \mathcal{P}'(k,2n), u \in W^P, \gamma \in W_{01}$ be in bijection with one another. Consider the Schubert variety $X^{\lambda} \subset  \IG(k,2n) $. Then 
\[  \Gamma_d(X^\lambda)=X^{\lambda^d}; \Gamma_d(X^\mu)=X^{\mu^d}; \Gamma_d(X^u)=X^{u^d}; \Gamma_d(X^\gamma)=X^{\gamma^d}.\]
\end{thm}

\subsection{Type B} As in type C, we restate Proposition \ref{recursion} for type B. Recall that we are considering the cases where $k<n$ for $\OG(k,2n+1)$. The set of positive roots is \[ \{t_i \pm t_j :1 \leq i < j \leq n \} \cup \{t_i :1 \leq i \leq n \}\] and the set of simple roots is \[\{t_i-t_{i+1}: 1 \leq i \leq n-1 \} \cup \{t_n \}.\] 
\begin{prop} \label{prop:brecursion} We restate the recursion for type B.
\begin{enumerate}
\item Consider the case $k=1$. 
\begin{enumerate}
\item The root $t_1+t_2$  is a maximal root of any degree $d>0$ and \[ z_1=s_{t_1+t_2}= (s_2s_1)(s_3s_2) \cdots (s_ns_{n-1})s_n(s_{n-2}s_{n-1})\cdots (s_{1}s_2).\] 
\item Let $u_B \in W^P$ and $d>0$. Then \begin{enumerate}
\item $z_d=z_1 \cdot \ldots \cdot z_1$;
\item $u_B^1$ is the minimal coset representative of $u_B\cdot z_1W_P$;
\item $u_B^d$ is the minimal coset representative of $u_B^{d-1} \cdot z_1 W_P$;
\item $u_B^d$ is the minimal coset representative of $u_B \cdot z_d W_P$.
\end{enumerate}
\end{enumerate}

\item Consider the case $2 \leq k \leq n-1$.  \begin{enumerate}
\item The root $t_1+t_{k+1}$ is a maximal root of degree $d=1$ and \[z_1=s_{t_1+t_{k+1}}=s_1\cdots s_{k-1}(s_{k+1}s_k) \cdots (s_ns_{n-1})s_n(s_{n-2}s_{n-1})\cdots (s_{k}s_{k+1}) s_{k-1}\cdots s_1;\]
\item The root $t_1+t_2$ is a maximal root of any degree $d>1$ and \[z_2=s_{t_1+t_2}=(s_2s_1)(s_3s_2) \cdots (s_ns_{n-1})s_n(s_{n-2}s_{n-1})\cdots (s_{1}s_2).\] 
\item Let $u_B\in W^{P}$. Then
\begin{enumerate}
\item If $d>0$ is even then $z_d=z_2 \cdot \ldots \cdot z_2$ ($d/2$-times);
\item If $d>0$ is odd then $z_d=z_2 \cdot \ldots \cdot z_2 \cdot z_1$ ($z_2$ appears $(d-1)/2$-times);
\item $u_B^1$ is the minimal coset representative of $u_B\cdot z_1W_P$;
\item $u_B^2$ is the minimal coset representative of $u_B\cdot z_2W_P$;
\item If $d>3$ is even then $u_B^d$ is the minimal coset representative of $u_B^{d-2} \cdot z_2W_P$;
\item If $d>2$ is odd then $u_B^d$ is the minimal coset representative of $u^{d-1}_B \cdot z_1W_P$;
\item For $d>0$, $u_B^d$ is the minimal coset representative of $u_B \cdot z_dW_P$.
\end{enumerate}
\end{enumerate}
\end{enumerate}
\end{prop}

\begin{remark}
Unlike the type A and C cases, it should be noted that $z_2 \neq z_1 \cdot z_1$ in the type B case.
\end{remark}

We will provide an example before we prove Proposition \ref{prop:brecursion}.
\begin{example}
Consider $u_B=12\bar435$ in the Weyl group of $\OG(3,11)$. Then the Hecke product $u_B \cdot z_1$ is
\begin{align*}
u_B \cdot z_1&= 12\bar435 \cdot s_1s_2s_4s_3s_5s_4s_5s_3s_4s_2s_1=  21\bar435 \cdot s_2s_4s_3s_5s_4s_5s_3s_4s_2s_1\\
&= 2\bar4135 \cdot s_4s_3s_5s_4s_5s_3s_4s_2s_1= 2\bar4153 \cdot s_3s_5s_4s_5s_3s_4s_2s_1 = 2\bar4513 \cdot s_5s_4s_5s_3s_4s_2s_1\\
&= 2\bar451\bar3 \cdot s_4s_5s_3s_4s_2s_1= 2\bar45\bar31 \cdot s_5s_3s_4s_2s_1=2\bar45\bar3\bar1 \cdot s_3s_4s_2s_1= 2\bar4\bar35\bar1 \cdot s_4s_2s_1\\
&= 2\bar4\bar3\bar15\cdot s_2s_1= 2\bar3\bar4\bar15\cdot s_1=\bar32\bar4\bar15.
\end{align*}

Notice that $u_B^1=(2<\bar3<\bar4|1<5)$ is the minimum length coset representative of $u_B \cdot z_1W_P$. We also have that Hecke product $u_B \cdot z_2$ is
\begin{align*}
u_B \cdot z_2&= 12\bar435 \cdot s_2s_1s_3s_2s_4s_3s_5s_4s_5s_3s_4s_2s_3s_1s_2=1\bar4235 \cdot s_1s_3s_2s_4s_3s_5s_4s_5s_3s_4s_2s_3s_1s_2\\
&= \bar41235 \cdot s_3s_2s_4s_3s_5s_4s_5s_3s_4s_2s_3s_1s_2=\bar41325 \cdot s_2s_4s_3s_5s_4s_5s_3s_4s_2s_3s_1s_2\\
&= \bar43125 \cdot s_4s_3s_5s_4s_5s_3s_4s_2s_3s_1s_2=\bar43152 \cdot s_3s_5s_4s_5s_3s_4s_2s_3s_1s_2\\
&= \bar43512 \cdot s_5s_4s_5s_3s_4s_2s_3s_1s_2=\bar4351\bar2 \cdot s_4s_5s_3s_4s_2s_3s_1s_2=\bar435\bar21 \cdot s_5s_3s_4s_2s_3s_1s_2\\
&= \bar435\bar2\bar1 \cdot s_3s_4s_2s_3s_1s_2=\bar43\bar25\bar1 \cdot s_4s_2s_3s_1s_2=\bar43\bar2\bar15 \cdot s_2s_3s_1s_2=\bar4\bar23\bar15 \cdot s_3s_1s_2\\
&= \bar4\bar2\bar135 \cdot s_1s_2=\bar2\bar4\bar135 \cdot s_2=\bar2\bar1\bar435.
\end{align*}

Notice that $u_B^2=(\bar4<\bar2<\bar1|3<5)$ is the minimum length coset representative of $u_B \cdot z_2W_P$. 
\end{example}

\begin{proof}
First observe that the highest root is $t_1+t_2$ and its coroot expands as a sum as simple coroots as follows: \[t_1+t_2=(t_1-t_2)+2(t_2-t_3)+\cdots+2(t_{n-1}-t_n)+2t_n. \] 

For part (1) (i.e. $k=1$) it follows that $\alpha:=t_1+t_2 \in R^+\backslash R_P^+$ is a maximal root for any $d>0$ since coefficient of the simple coroot $t_1-t_2$ in the expansion of $t_1+t_2$ is 1. By Proposition \ref{recursion}, $s_{\alpha} \cdot z_{d-\alpha^\vee}=z_dW_P$, $z_1=s_{\alpha}$, and $z_d=z_1 \cdot \ldots \cdot z_1$. A direct calculation shows that $u_B^1$ is the minimal length coset representative of $u_B \cdot z_1W_P$. Next, part 1 of Definition \ref{defn:typeBperm} implies that $u_B^d=(u_B^{d-1})^1$. So, $u_B^d$ is the minimal length coset representative of $u_B^{d-1} \cdot z_1 W_P$. The result follows for part 1.

For part (2) (i.e. $1<k<n$) it follows that $\alpha=t_1+t_2 \in R^+\backslash R_P^+$ is a maximal root for any $d>1$ since coefficient of the simple coroot $t_k-t_{k+1}$ in the expansion of $t_1+t_2$ is 2. By Proposition \ref{recursion}, $s_{\alpha} \cdot z_{d-\alpha^\vee}=z_dW_P$ and $z_2=s_{\alpha}$. A direct calculation shows that $u_B^2$ is the minimal length coset representative of $u_B \cdot z_2W_P$. When $d>0$ is even, part 2b of Definition \ref{defn:typeBperm} implies that $u_B^d=(u_B^{d-2})^2$ since we know that $u_C^d=(u_C^{d-2})^2$ and $u_B^d=u_C^d$. So, $u_B^d$ is the minimal length coset representative of $u_B^{d-2} \cdot z_2 W_P$.

We will now consider the $d=1$ for part (2). Observe that \[t_1+t_{k+1}=(t_1-t_2)+\cdots+(t_k-t_{k+1})+2(t_{k+1}-t_{k+2})+\cdots+2t_n. \] Every root $\beta \in R^+$ such that $\beta>t_1+t_{k+1}$ takes on the form $\beta=t_1+t_j$ where $2 \leq j \leq k$. But the coroot of $t_1+t_j$ written as sum of simple coroots is \[t_1+t_{j}=(t_1-t_2)+\cdots+(t_{j-1}-t_j)+2(t_j-t_{j+1})+\cdots+2(t_k-t_{k+1})+2(t_{k+1}-t_{k+1})+\cdots+2t_n. \] It follows that $\delta:=t_1+t_{k+1} \in R^+\backslash R_P^+$ is a maximal root of $d=1$ by consider the coefficients of $t_k-t_{k+1}$ in the simple coroot expansion of $t_1+t_{k+1}$ and $t_1+t_j$. By Proposition \ref{recursion}, $s_{\delta} \cdot z_{d-\delta^\vee}=z_dW_P$ and $z_1=s_{\delta}$. A direct calculation shows that $u_B^1$ is the minimal length coset representative of $u_B \cdot z_1W_P$. When $d>0$ is odd, part 2c of Definition \ref{defn:typeBperm} indicates that $u_B^d=(u_B^{d-1})^1$. So, $u_B^d$ is the minimal length coset representative of $u_B^{d-1} \cdot z_1 W_P$. Finally, Proposition \ref{recursion} implies that if $d>0$ is even then $z_d=z_2 \cdot \ldots \cdot z_2$ ($d/2$-times). Similarly, if $d>0$ is odd then $z_d=z_2 \cdot \ldots \cdot z_2 \cdot z_1$ ($z_2$ appears $(d-1)/2$-times). Part 2 follows. This completes the proof. \end{proof}



The next theorem describes the curve neighborhoods of Schubert varieties in $\OG(k,2n+1)$.

\begin{thm} \label{Thm:Bcrv}
Let $\lambda \in \Lambda, \mu \in \mathcal{P}'(k,2n), u \in W^P, \gamma \in W_{01}$ be in bijection with one another. Consider the Schubert variety $X^{\lambda} \subset  \OG(k,2n+1) $. Then 
\[ \Gamma_d(X^\lambda)=X^{\lambda^d}; \Gamma_d(X^\mu)=X^{\mu^d}; \Gamma_d(X^u)=X^{u^d}; \Gamma_d(X^\gamma)=X^{\gamma^d}.\]
\end{thm}




\section{Minimum quantum degree calculation for types B and C Grassmannians}
\label{sect:typeCpostnikov}

The purpose of this section is to prove Theorem \ref{thm:visthm}. It's important to note for this section that the Weyl group in types B and C is the hyperoctohedral group on $2n$ elements. So, $W_C^P$ is the set of minimal length representatives of the hyperoctohedral group. Moreover, the hyperoctohendral group is a subgroup of the symmetric group on $2n$ elements $S_{2n}$. In this section we will use subscripts to indicate whether an element or an order is to be considered in type A or type C. For example, let $\leq_C$ denote the Bruhat order on $W_C$ in type $C$ and let $\leq_A$ denote the Bruhat order on $W_A:=S_{2n}$ in type A. We begin with a lemma and a proposition regarding the correspondence between the Bruhat order and partition inclusion on partitions in the set $\mathcal{P}'(k,2n)$. The next lemma is proved in \cite[Corollary 8.1.9]{BjornerBrenti}.

\begin{lemma} \label{lem:order}
Let $v,w \in W_C \subset W_A$. Then $v \leq_C w$ if and only if $v \leq_A w$. In particular, for any $v,w \in W_C^P$, \[v=(v(1)<v(2)<\cdots<v(k)) \leq w=(w(1)<w(2)<\cdots<w(k)) \] if and only if $v_i \leq w_i$ for all $1 \leq i \leq k$.
\end{lemma}

As an immediate application of Lemma \ref{lem:order} we have the following useful statement regarding partition inclusion and the Bruhat order in types B and C.
\begin{prop} \label{prop:partin}
Let $X \in \{ \OG(k,2n+1), \IG(k,2n)\}$. We have $X^\mu \subset X^\lambda \subset X$ if and only if $\lambda \subset \mu$ for all $\lambda, \mu \in \mathcal{P}'(k,2n)$. That is, partition inclusion in $\mathcal{P}'(k,2n)$ corresponds to the Bruhat order in types B and C.
\end{prop}

\begin{example} 
For an example of the Bruhat order corresponding to partition inclusion consider the case $n=5$ and $k=3$. First observe that
\[v_C=(1<5<\bar{4}|2<3) \in W_C^P\] corresponds to \[v_A=(1<5<\bar{4}|2<3<4<\bar{5}<\bar{3}<\bar{2}<\bar{1}) \in W_A^P.\] Moreover, both permutations correspond to the 01-word 1000101000 and the partition $\mu \in \mathcal{P}'(3,10) \subset \mathcal{P}(3,10)$ given by \[ \mu=\yng(7,4,3). \] Next observe that \[w_C=(1<\bar{4}<\bar{2}|3<5) \in W_C^P\] corresponds to \[w_A=(1<\bar{4}<\bar{2}|2<3<4<5<\bar{5}<\bar{3}<\bar{1}) \in W_A^P.\] Moreover, both permutations correspond to the 01-word 1000001010 and the partition $\lambda \in \mathcal{P}'(3,10) \subset \mathcal{P}(3,10)$ given by \[ \lambda=\yng(7,2,1). \] We see that $v_C \leq w_C$, $v_A \leq w_A$, and $\lambda \subset \mu$.
\end{example}



\subsection{Pictorial formulas}
The first theorem of this subsection establishes that the minimum degree that appears in the quantum product of Schubert classes are computed using Young diagrams.
\begin{thm} \label{thm:visthmb}
Let $X \in \{ \OG(k,2n+1), \IG(k,2n)\}$. For any pair of Schubert classes  $[X^\lambda] \star [X_\mu] \in \QH^*(X)$ where $\lambda, \mu \in \mathcal{P}'(k,2n)$, the smallest degree $d$ such that $q^d$ appears in  $[X^\lambda] \star [X_\mu]$ with nonzero coefficient is the smallest integer $d$ such that $\lambda^d \subset \mu$.
\end{thm}

\begin{example}
Recall Example \ref{ex:C}. For $\IG(5,2 \cdot 8)$ consider $\lambda=(11 \geq 11 \geq 11 \geq 4 \geq 4) \in \mathcal{P}'(5,2 \cdot 8)$ and $\mu=(7,7) \in  \mathcal{P}'(5,2\cdot 8).$ We have that 3 is the smallest degree $d$ such that $q^d$ appears in $[X^\lambda] \star [X_\mu]$ with nonzero coefficient.
\end{example}

\begin{example}
For $\OG(5,2 \cdot 8+1)$ consider $\lambda=(11 \geq 11 \geq 11 \geq 4 \geq 4) \in \mathcal{P}'(5,2 \cdot 8)$ and $\mu=(7,7) \in  \mathcal{P}'(5,2 \cdot 8).$ Then we have that 
\[ \lambda=\yng(11,11,11,4,4) \mbox{ and } \mu=\yng(7,7)\]
\begin{align*}
\lambda^2 &= \yng(9,2,2)\\
\lambda^3 &= \yng(1,1,1)\\
\lambda^4 &= \emptyset.
\end{align*}
Thus, 4 is the smallest degree $d$ such that $q^d$ appears in $ [X^\lambda] \star [X_\mu]$ with nonzero coefficient.
\end{example}

\section{Technical Proofs}\label{sect:techproof}

\subsection{Type C} 

We begin with a technical lemma to prove Proposition \ref{prop:C}. Let $\lambda_C \in \Lambda_C$, $\mu_C \in \mathcal{P}'(k,n)$, $u_C \in W^{P}$, and $\gamma_C \in W_{01}$ are in bijection with one another and $d \geq 0$. The permutation $u^d_C \in W^{P}$, the word $\gamma^d_C \in W_{01}$, and the partition $\mu^d_C \in \mathcal{P}'(k,n)$ are in bijection with each other by a direct application of Propositions \ref{prop:crecursion} and \ref{prop:arecursion}. We will complete the proof by showing that $\lambda^1_C \in \Lambda_C$ and $u^1_C \in W^{P}$ are in bijection.

In this section let $|i|=\min \{i,2n+1-i \}$.

\begin{lemma} \label{lem:backend}
Let $u\in W^P$ where $u \mapsto \lambda \in \Lambda$. Suppose that $u(m)=\overline{\alpha}$ where $1 \leq \alpha \leq n$ and $\{1,2,\cdots, \alpha \} \subset \{|u(1)|, |u(2)|, \cdots, |u(k)| \}$. Then $\lambda_m=0$.
\end{lemma}

\begin{proof}
By a direct calculation we have $\#\{ i<m: w(i)+\overline{\alpha}>2n+1 \}=k-\alpha$. Then we have that
\begin{align*}
\lambda_m&= 2n+1-k-w(m)+\# \{ i<m: w(i)+w(j)>2n+1 \}\\
&= \alpha-k+\# \{ i<m: w(i)+\overline{\alpha}>2n+1 \}\\
&= 0.
\end{align*}\end{proof}

Next is the proof of Prop \ref{prop:C}.

\begin{proof}
If $\lambda_1+\lambda_j > 2(n-k)+j-1$ for all $2 \leq j \leq k$ then it's straight forward to show that $v^1\mapsto \lambda^1$.

Otherwise, find the smallest $j$ such that $\lambda_1+\lambda_j \leq 2(n-k)+j-1$ and consider 
\[ 
\lambda^1=(\lambda_2 \geq \lambda_3 \geq \cdots \geq \lambda_{j-1} \geq \lambda_j-1 \geq \cdots \geq \lambda_{k-1}-1 \geq 0) \in  \Lambda
\]
 where $-1$'s are replaced by 0. Let \[ v=(v(1)<v(2)<\cdots<v(m)<v(m+1)<\cdots<v(k)) \in W^{P} \] where $v(k-i)=\overline{1+i}$ for $0 \leq i \leq k-m-1$ and $v(m)<\overline{k-m+1}$ and 
\[ 
v^1=(v(2)<\cdots<v(m)<\overline{1+k-m}<v(m+1)<\cdots<v(k)) \in W^{P}.
\]
It's clear that $v^1(l-1) =v(l) \mapsto \lambda_{l}$ for $2 \leq l \leq j-1$.

Let $l$ be an integer where $j \leq l \leq m$ and $\lambda_l \geq 1$. Then we have the following
\begin{align*}
v^1(l-1)=v(l)&= 2n+1-k-\lambda_l+\# \{ i \leq l: \lambda_i+\lambda_l \leq 2(n-k)+l-i \}\\
&= 2n+1-k-\lambda_l+1+\# \{ 2 \leq i \leq l: \lambda_i+\lambda_l \leq 2(n-k)+l-i \}\\
&= 2n+1-k-(\lambda_l-1)+\# \{ 2 \leq i \leq l: \lambda_i+\lambda_l \leq 2(n-k)+l-i \}
\end{align*}
So $v^1(l-1) \mapsto \lambda_l-1$ for $j \leq l \leq m$ and $\lambda_l \geq 1$.

Let $l$ be an integer where $j \leq l \leq m$ and $\lambda_l =0$. Observe the following
\begin{align*}
\lambda_l&= 2n+1-k-v(l)+\# \{i<l: v(i)+v(l)>2n+1 \}\\
& \geq 2n+1-k-v(l)+\# \{2 \leq i<l: v(i)+v(l)>2n+1 \}\\
&= 2n+1-k-v^1(l-1)+\# \{ i<l: v^1(i)+v^1(l-1)>2n+1 \}\\
&=\lambda^1_{l-1}.
\end{align*}
Therefore, $\lambda_l=\lambda^1_{l-1}=0$ and $v^1(l-1) \mapsto 0=\lambda_l$ for $j \leq l \leq m$ and $\lambda_l = 0$.

By Lemma \ref{lem:backend}, $\lambda_l=0$ for $m+1 \leq l \leq k$ and $v^1(m) \mapsto 0$ for $m \leq l \leq k$. So $v^1(l-1) \mapsto 0=\lambda_l$ for all $m+1 \leq l \leq k$. Observe that $v^1(k)=\bar{1} \mapsto 0=\lambda^1_k$. The result follows.
\end{proof}

\subsection{Type B}Next is the argument of Proposition \ref{prop:B}. First observe that this is clear for the $k=1$ case. The cases where $\lambda_B^d=\lambda_C^d$ have already been established by the definitions of $u_C^d$ and $u_B^d$ and $\lambda_C^d$ and $\lambda_B^d$ when $d>1$ and even. We now must consider the case where $d=1$.

Let $\lambda_B \in \Lambda_B$, $\mu_B \in \mathcal{P}'(k,n)$, $u_B \in W^{P}$, and $\gamma_B \in W_{01}$ are in bijection. The permutation $u^1_B \in W^{P}$, the word $\gamma^1_B \in W_{01}$, and the partition $\mu^1_B \in \mathcal{P}'(k,n)$ are in bijection with each other by a direct application of Proposition \ref{prop:brecursion}. We will complete the proof by showing that 
\begin{enumerate}
\item The partition $\lambda^1_B \in \Lambda_C$ and $u^1_C \in W^{P}$ are in bijection;
\item The partition $\mu^1_B \in \mathcal{P}'(k,n)$ and $\gamma^1_B \in W_{01}$ are in bijection.
\end{enumerate}

We begin by showing the first listed bijection holds.

\subsubsection{The partition $\lambda^1_B \in \Lambda_C$ and $u^1_C \in W^{P}$ are in bijection}
\begin{prop}
Let $d=1$ and $\lambda_1-\ell(\lambda_B) \geq 2(n-k)+1$ then $u_B^1$ is in bijection with $\lambda_C^1$ with all but the last $(2n-k-\lambda_1)$ 0's interchanged with 1's.
\end{prop}

\begin{proof}
First note that $\lambda_1-\ell(\lambda_B) \geq 2(n-k)+1$ is equivalent to $k-\ell(\lambda_B) \geq 2n+1-k-\lambda_1$
Let $\lambda_B \mapsto u_B$. By the bijection we have that $2n+1-k-\lambda_1=u_B(1)$. It follows that $\{1,2, \cdots, u_B(1) \} \subset \{ |u_B(1)|, |u_B(2)|, \cdots, |u_B(k)| \}.$ Thus, $\{1,2, \cdots, u_B(1)-1 \} \subset \{ |u_B^1(1)|, |u_B^1(2)|, \cdots, |u_B^1(k)| \}.$ In particular, the last $u_B(1)-1=2n-k-\lambda_1$ parts of $\lambda_B^1$ are zero.

Let $m$ be the largest integer such that $\lambda_m>0$. By Lemma \ref{lem:backend} there exists $j \notin \{ |u(1)|, |u(2)|,\cdots, |u(k)| \}$ for some $u(m) < j \leq \bar{1}$ and $1 \leq |j| < |u(m)|$. Let $J$ be the largest such $j$. Next observe that
\[
u_C^d:=\left(\widehat{u(1)}<u(2)< \cdots <u(m_1)<u(m_2)<\cdots<u(k)< \bar{1} \right)
\] 
and  \[ u_B^1:=\left(\widehat{u(1)}<u(2)<u(3)< \cdots<u(m)<\cdots<u(m_1)<J<u(m_2) <\cdots<\widehat{ \overline{u(1)}} \cdots<u(k) \right).\] Also, $\{1, 2,  \cdots, \widehat{u(1)}, \cdots, |J| \} \subset \{|u(1)|, |u(2)|, \cdots, |u(k)| \}$ and $u^1_B(i)=u^1_C(i)$ for $1 \leq i \leq m_1$. The second inequality indicates that $(\lambda_B^1)_i=(\lambda_C^1)_i$ for $1 \leq i \leq m_1$ and nonzero. 

By Lemma \ref{lem:backend} we have that $(\lambda_C^1)_i=0$ for $m_2\leq i \leq k$. Consider $(\lambda_B^1)_{i_0}$ for some $i_0, m_2 \leq i_0 \leq k-(2n-k-\lambda_1)$.

By a direct calculation we have $\#\{ i<i_0: w(i)+w(i_0)>2n+1 \}=k-|w(i_0)|+1$. Then we have that
\begin{align*}
\lambda_{i_0}&= 2n+1-k-w(i_0)+\# \{ i<m: w(i)+w(i_0)>2n+1 \}\\
&= |w(i_0)|-k+\# \{ i<m: w(i)+w(i_0)>2n+1 \}\\
&= 1.
\end{align*}
The result follows.
\end{proof}

\begin{prop}
Let $d=1$ and $\lambda_1-\ell(\lambda_B) \leq 2(n-k)$. Then $u_B^1$ is in bijection with $\lambda_C^1$.
\end{prop}

\begin{proof}
First note that $\lambda_1-\ell(\lambda_B) \leq 2(n-k)$ is equivalent to $k-\ell(\lambda_B) \leq 2n-k-\lambda_1$.
Let $\lambda_B \mapsto u_B$. By the bijection we have that $2n-k-\lambda_1=u_B(1)-1$. In particular, by \ref{lem:backend}, $\overline{u_B(1)} \notin \{ |u(1)|,|u(2)|, \cdots, |u(k)| \}$. Thus $u_B^1=u_C^1$. The result follows.
\end{proof}

\subsubsection{The partition $\mu^1_B \in \mathcal{P}'(k,2n)$ and $\gamma^1_B \in W_{01}$ are in bijection.} 

\begin{lemma}
The partition $\mu^1_B \in \mathcal{P}'(k,2n)$ and $\gamma^1_B \in W_{01}$ are in bijection.
\end{lemma}

\begin{proof}
Let $\mu \in  \mathcal{P}'(k,n)$ be a $m$-wingtip symmetric partition. That is, $m$ is the largest nonnegative integer such that $|\gamma(i)-\gamma(2n-i)|=1$ for $0 \leq i \leq m$. So we have that \[\gamma_B=\gamma(1) \cdots \gamma(m) \gamma(m+1) \gamma(m+2) \cdots \gamma(2n). \] Suppose that $\gamma(\overline{m})=0$. Then $\gamma(\overline{m+1})=0$, otherwise $\mu_B$ is not $m$-wingtip symmetric. So we have that 
\begin{align*}
\gamma_B&= \overbrace{0 \cdots 010 \cdots 01 \cdots \gamma(\overline{m+2})0}^{ \mbox{First $2n-m$ characters} }0 \gamma(\overline{m-1}) \cdots \gamma(\overline{1});\\
\gamma_B^1&= \overbrace{0 \cdots 0\hat{1}0 \cdots 01 \cdots \gamma(\overline{m+2})1}^{ \mbox{First $2n-m$ characters} }0 \gamma(\overline{m-1}) \cdots \gamma(\overline{1});.
\end{align*}
Observe that the first $2n-m$ characters correspond to curve neighborhoods in $\Gr(k,2n-m)$ for some $k$. Thus one shifts the permutation up and deletes a box out of the corresponding rows. Notice that the last rows corresponding to the last $m$ characters of the word will remain the same. The introduction of the character 1 in the $\overline{m+1}$ position forces a row to be of length $j$ (the column before the boundary edge $\gamma(\overline{m})=0$ is in).

Suppose that $\gamma(\overline{m})=1$. Then $\gamma(\overline{m+1})=0$, otherwise $\mu_B$ is not $m$-wingtip symmetric. So we have that 
\begin{align*}
\gamma_B&= \overbrace{0 \cdots 010 \cdots 01 \cdots \gamma(\overline{m+2})0}^{ \mbox{First $2n-m$ characters} }1 \gamma(\overline{m-1}) \cdots \gamma(\overline{1});\\
\gamma_B^1&= \overbrace{0 \cdots 0\hat{1}0 \cdots 01 \cdots \gamma(\overline{m+2})1}^{ \mbox{First $2n-m$ characters} }1 \gamma(\overline{m-1}) \cdots \gamma(\overline{1});.
\end{align*}
Observe that the first $2n-m$ characters correspond to curve neighborhoods in $\Gr(k,2n-m)$ for some $k$. Thus one shifts the permutation up and deletes a box out of the corresponding rows. Notice that the last rows corresponding to the last $m$ characters of the word will remain the same. The introduction of the character 1 in the $\overline{m+1}$ position forces a row to be of length $\mu_i$ (the row that boundary edge $\gamma(\overline{m})=1$ is in).
\end{proof}

Proposition \ref{prop:B} follows.

\section{No conflict of interest statement}
On behalf of all authors, the corresponding author states that there is no conflict of interest.

\bibliographystyle{halpha}
\bibliography{bibliography}

\begin{bibdiv}
\begin{biblist}

\bib{Bar}{misc}{
      author={B\"arligea, Christoph},
       title={Curve neighborhoods and minimal degrees in quantum products},
        date={2016},
}

\bib{Belkale}{article}{
      author={Belkale, P.},
       title={Transformation formulas in quantum cohomology},
        date={2004},
     journal={Compositio Mathematica},
      volume={140},
      number={3},
       pages={778\ndash 792},
}

\bib{BjornerBrenti}{book}{
      author={Bjorner, A.},
      author={Brenti, F.},
       title={Combinatorics of coxeter groups},
   publisher={Springer},
        ISBN={978-3-540-44238-7},
}

\bib{BCMP:eqkt}{article}{
      author={Buch, Anders},
      author={Chaput, Pierre-Emmanuel},
      author={Mihalcea, Leonardo~C.},
      author={Perrin, Nicolas},
       title={A {C}hevalley formula for the equivariant quantum {$K$}-theory of
  cominuscule varieties},
     journal={preprint, available at https:},
        ISSN={/arxiv.or},
}

\bib{BCMP:qkfin}{article}{
      author={Buch, Anders},
      author={Chaput, Pierre-Emmanuel},
      author={Mihalcea, Leonardo~C.},
      author={Perrin, Nicolas},
       title={Finiteness of cominuscule quantum {K}-theory},
        date={2013},
     journal={Annales Sci. de L'\'Ecole Normale Sup\'erieure},
      number={46},
       pages={477\ndash 494},
}

\bib{BCLM:EulerChar}{misc}{
      author={Buch, Anders~S.},
      author={Chung, Sjuvon},
      author={Li, Changzheng},
      author={Mihalcea, Leonardo~C.},
       title={Euler characteristics in the quantum $k$-theory of flag
  varieties},
        date={2019},
}

\bib{buch.m:nbhds}{article}{
      author={Buch, Anders~S.},
      author={Mihalcea, Leonardo~C.},
       title={Curve neighborhoods of {S}chubert varieties},
        date={2015},
        ISSN={0022-040X},
     journal={J. Differential Geom.},
      volume={99},
      number={2},
       pages={255\ndash 283},
         url={http://projecteuclid.org/euclid.jdg/1421415563},
      review={\MR{3302040}},
}

\bib{Buch}{article}{
      author={Buch, Anders~Skovsted},
       title={Quantum cohomology of {G}rassmannians},
        date={2003},
        ISSN={0010-437X},
     journal={Compositio Math.},
      volume={137},
      number={2},
       pages={227\ndash 235},
         url={http://dx.doi.org/10.1023/A:1023908007545},
      review={\MR{1985005}},
}

\bib{BKT2}{article}{
      author={Buch, Anders~Skovsted},
      author={Kresch, Andrew},
      author={Tamvakis, Harry},
       title={Quantum {P}ieri rules for isotropic {G}rassmannians},
        date={2009},
        ISSN={0020-9910},
     journal={Invent. Math.},
      volume={178},
      number={2},
       pages={345\ndash 405},
         url={http://dx.doi.org/10.1007/s00222-009-0201-y},
      review={\MR{2545685}},
}

\bib{BKT}{article}{
      author={{Buch, A. and Kresch, A. and Tamvakis, H.}},
       title={{Quantum Giambelli formulas for isotropic Grassmannians}},
        date={{2012}},
        ISSN={{0025-5831}},
     journal={{Math. Ann.}},
      volume={{354}},
      number={{3}},
       pages={{801\ndash 812}},
      review={{\MR{2983068}}},
}

\bib{CMP}{article}{
      author={Chaput, P.-E.},
      author={Manivel, L.},
      author={Perrin, N.},
       title={Quantum cohomology of minuscule homogeneous spaces},
        date={2008},
     journal={Transform. Groups},
      volume={13},
      number={1},
       pages={47\ndash 89},
}

\bib{FW}{article}{
      author={Fulton, W.},
      author={Woodward, C.},
       title={On the quantum product of {S}chubert classes},
        date={2004},
        ISSN={1056-3911},
     journal={J. Algebraic Geom.},
      volume={13},
      number={4},
       pages={641\ndash 661},
         url={http://dx.doi.org/10.1090/S1056-3911-04-00365-0},
      review={\MR{2072765}},
}

\bib{postnikov:affine}{article}{
      author={Postnikov, Alexander},
       title={Affine approach to quantum {S}chubert calculus},
        date={2005},
        ISSN={0012-7094},
     journal={Duke Math. J.},
      volume={128},
      number={3},
       pages={473\ndash 509},
         url={https://doi.org/10.1215/S0012-7094-04-12832-5},
      review={\MR{2145741}},
}

\bib{postnikov:qbruhat}{article}{
      author={Postnikov, Alexander},
       title={Quantum {B}ruhat graph and {S}chubert polynomials},
        date={2005},
        ISSN={0002-9939},
     journal={Proc. Amer. Math. Soc.},
      volume={133},
      number={3},
       pages={699\ndash 709},
         url={https://doi.org/10.1090/S0002-9939-04-07614-2},
      review={\MR{2113918}},
}

\bib{yong}{article}{
      author={Yong, A.},
       title={Degree bounds in quantum {S}chubert calculus},
        date={2003},
     journal={Proceedings of the AMS},
      volume={131},
      number={9},
       pages={2649\ndash 2655},
}

\end{biblist}
\end{bibdiv}

\end{document}